\documentclass{amsart}
\usepackage{amsmath}
\usepackage{amssymb}
\usepackage{amsfonts}

\newtheorem{theorem}{Theorem}
\theoremstyle{plain}

\newtheorem{corollary}{Corollary}

\numberwithin{equation}{section}

\input{tcilatex}

\begin{document}
\title[On new inequalities]{On fractional inequalities via Montgomery
identities integrals }
\author{Mehmet Zeki Sar\i kaya$^{\star }$}
\address{Department of Mathematics, Faculty of Science and Arts, D\"{u}zce
University, D\"{u}zce, Turkey}
\email{sarikayamz@gmail.com}
\thanks{$^{\star }$corresponding author}
\author{Hatice YALDIZ}
\email{yaldizhatice@gmail.com}
\author{Erhan SET}
\email{erhanset@yahoo.com}

\begin{abstract}
In the present work we give several new integral inequalities of the type
Riemann-Liouville fractional integral via Montgomery identities integrals.
\end{abstract}

\maketitle

\section{Introduction}

The inequality of Ostrowski \cite{Ostrowski} gives us an estimate for the
deviation of the values of a smooth function from its mean value. More
precisely, if $f:[a,b]\rightarrow \mathbb{R}$ is a differentiable function
with bounded derivative, then%
\begin{equation*}
\left\vert f(x)-\frac{1}{b-a}\int\limits_{a}^{b}f(t)dt\right\vert \leq \left[
\frac{1}{4}+\frac{(x-\frac{a+b}{2})^{2}}{(b-a)^{2}}\right] (b-a)\left\Vert
f^{\prime }\right\Vert _{\infty }
\end{equation*}
for every $x\in \lbrack a,b]$. Moreover the constant $1/4$ is the best
possible.

For some generalizations of this classic fact see the book \cite[p.468-484]%
{mitrovich} by Mitrinovic, Pecaric and Fink. A simple proof of this fact can
be \ done by using the following identity \cite{mitrovich}:

If $f:[a,b]\rightarrow \mathbb{R}$ is differentiable on $[a,b]$ with the
first derivative $f^{\prime }$ integrable on $[a,b],$ then Montgomery
identity holds:%
\begin{equation}
f(t)=\frac{1}{b-a}\int\limits_{a}^{b}f(s)ds+\int\limits_{a}^{b}P(t,s)f^{%
\prime }(s)ds,  \label{h}
\end{equation}%
where $P(x,t)$ is the Peano kernel defined by%
\begin{equation}
P(t,s):=\left\{ 
\begin{array}{ll}
\dfrac{s-a}{b-a}, & a\leq s<t \\ 
&  \\ 
\dfrac{s-b}{b-a}, & t\leq s\leq b.%
\end{array}%
\right.  \label{hh}
\end{equation}%
Suppose now that $w:\left[ a,b\right] \rightarrow \lbrack 0,\infty )$ is
some probability density function, i.e. it is a positive integrable function
satisfying $\int_{a}^{b}w\left( t\right) dt=1$,\ and $W\left( t\right)
=\int_{a}^{t}w\left( x\right) dx$ for $t\in \left[ a,b\right] $, $W\left(
t\right) =0$ for $t<a$ and $W\left( t\right) =1$ for $t>b$. The following
identity (given by Pe\v{c}ari\'{c} in \cite{pecaric}) is the weighted
generalization of the Montgomery identity:%
\begin{equation}
f\left( x\right) =\int\limits_{a}^{b}w\left( t\right) f\left( t\right)
dt+\int\limits_{a}^{b}P_{w}\left( x,t\right) f^{^{\prime }}\left( t\right) dt%
\text{,}  \label{hhh}
\end{equation}%
where the weighted Peano kernel is%
\begin{equation*}
P_{w}(x,t):=\left\{ 
\begin{array}{ll}
W\left( t\right) , & a\leq t<x \\ 
&  \\ 
W\left( t\right) -1, & x\leq t\leq b.%
\end{array}%
\right.
\end{equation*}

The Riemann-Liouville fractional integral operator of order $\alpha \geq 0$
with $a\geq 0$ is defined by%
\begin{eqnarray}
J_{a}^{\alpha }f(x) &=&\frac{1}{\Gamma (\alpha )}\int\limits_{a}^{x}(x-t)^{%
\alpha -1}f(t)dt,  \label{a} \\
J_{a}^{0}f(x) &=&f(x).  \notag
\end{eqnarray}%
Recently, many authors have studied a number of inequalities by used the
Riemann-Liouville fractional integrals, see (\cite{Belarbi}-\cite{gorenflo}, 
\cite{sarikaya}, \cite{sarikaya1}) and the references cited therein. More
details, for necessary definitions and mathematical preliminaries of
fractional calculus theory, one can consult \cite{gorenflo}, \cite{samko}.

\section{Results}

\begin{theorem}
\label{3m} Let $f:\left[ a,b\right] \rightarrow \mathbb{R}$ be a
differentiable function on $\left[ a,b\right] $ such that $f^{\prime }\in
L_{p}\left[ a,b\right] $ with $\frac{1}{p}+\frac{1}{q}=1,\ p>1,$ and $\alpha
\geq 0$. Then, the following inequality holds:%
\begin{equation}
\left\vert \Gamma \left( \alpha +1\right) J_{a}^{\alpha }f\left( b\right)
-\left( b-a\right) ^{\alpha -1}\int\limits_{a}^{b}f\left( s\right)
ds\right\vert \leq \left( b-a\right) ^{\alpha +\frac{1}{q}}\left( \frac{1}{%
\left( \alpha q+1\right) ^{\frac{1}{q}}}+\frac{1}{\left( q+1\right) ^{\frac{1%
}{q}}}\right) \left\Vert f^{^{\prime }}\right\Vert _{p}.  \label{9}
\end{equation}
\end{theorem}

\begin{proof}
We can write the Riemann-Liouville fractional integral operator as follows:%
\begin{equation}
\Gamma \left( \alpha \right) J_{a}^{\alpha }f\left( b\right)
=\int\limits_{a}^{b}\left( b-t\right) ^{\alpha -1}f\left( t\right) dt.
\label{1}
\end{equation}%
Thus, using Montgomery identity in (\ref{1}), we have%
\begin{eqnarray}
\Gamma \left( \alpha \right) J_{a}^{\alpha }f\left( b\right) 
&=&\int\limits_{a}^{b}\left( b-t\right) ^{\alpha -1}\left[ \frac{1}{b-a}%
\int\limits_{a}^{b}f\left( s\right) ds+\int\limits_{a}^{b}P\left( t,s\right)
f^{^{\prime }}\left( s\right) ds\right] dt  \notag \\
&&  \label{2} \\
&=&\frac{1}{b-a}\int\limits_{a}^{b}\left( b-t\right) ^{\alpha -1}\left[
\int\limits_{a}^{b}f\left( s\right) ds+\int\limits_{a}^{t}\left( s-a\right)
f^{^{\prime }}\left( s\right) ds+\int\limits_{t}^{b}\left( s-b\right)
f^{^{\prime }}\left( s\right) ds\right] dt.  \notag
\end{eqnarray}%
By an interchange of the order of integration, we get%
\begin{equation}
\int\limits_{a}^{b}\left( b-t\right) ^{\alpha -1}\left(
\int\limits_{a}^{b}f\left( s\right) ds\right) dt=\frac{\left( b-a\right)
^{\alpha }}{\alpha }\int\limits_{a}^{b}f\left( s\right) ds,  \label{3}
\end{equation}%
\begin{equation}
\int\limits_{a}^{b}\left( b-t\right) ^{\alpha -1}\left(
\int\limits_{a}^{t}\left( s-a\right) f^{^{\prime }}\left( s\right) ds\right)
dt=\frac{b-a}{\alpha }\int\limits_{a}^{b}\left( b-s\right) ^{\alpha
}f^{^{\prime }}\left( s\right) ds-\frac{1}{\alpha }\int\limits_{a}^{b}\left(
b-s\right) ^{\alpha +1}f^{^{\prime }}\left( s\right) ds,  \label{4}
\end{equation}%
\begin{equation}
\int\limits_{a}^{b}\left( b-t\right) ^{\alpha -1}\left(
\int\limits_{t}^{b}\left( s-b\right) f^{^{\prime }}\left( s\right) ds\right)
dt=\frac{1}{\alpha }\int\limits_{a}^{b}\left( b-s\right) ^{\alpha
+1}f^{^{\prime }}\left( s\right) ds-\frac{\left( b-a\right) ^{\alpha }}{%
\alpha }\int\limits_{a}^{b}\left( b-s\right) f^{^{\prime }}\left( s\right)
ds.  \label{5}
\end{equation}%
Thus, using (\ref{3}), (\ref{4}) and (\ref{5}) in (\ref{2}) we get 
\begin{eqnarray}
&&\Gamma \left( \alpha +1\right) J_{a}^{\alpha }f\left( b\right) -\left(
b-a\right) ^{\alpha -1}\int\limits_{a}^{b}f\left( s\right) ds  \notag \\
&&  \label{z} \\
&=&\int\limits_{a}^{b}\left( b-s\right) ^{\alpha }f^{^{\prime }}\left(
s\right) ds-\left( b-a\right) ^{\alpha -1}\int\limits_{a}^{b}\left(
b-s\right) f^{^{\prime }}\left( s\right) ds\text{, }\alpha \geq 0.  \notag
\end{eqnarray}%
By taking the modulus and applying H\"{o}lder inequality, we have%
\begin{eqnarray*}
&&\left\vert \Gamma \left( \alpha +1\right) J_{a}^{\alpha }f\left( b\right)
-\left( b-a\right) ^{\alpha -1}\int\limits_{a}^{b}f\left( s\right)
ds\right\vert  \\
&& \\
&\leq &\left( \int\limits_{a}^{b}\left\vert f^{^{\prime }}\left( s\right)
\right\vert ^{p}ds\right) ^{\frac{1}{p}}\left( \int\limits_{a}^{b}\left(
b-s\right) ^{\alpha q}ds\right) ^{\frac{1}{q}} \\
&& \\
&&+\left( b-a\right) ^{\alpha -1}\left( \int\limits_{a}^{b}\left\vert
f^{^{\prime }}\left( s\right) \right\vert ^{p}ds\right) ^{\frac{1}{p}}\left(
\int\limits_{a}^{b}\left( b-s\right) ^{q}ds\right) ^{\frac{1}{q}} \\
&& \\
&=&\left( b-a\right) ^{\alpha +\frac{1}{q}}\left( \frac{1}{\left( \alpha
q+1\right) ^{\frac{1}{q}}}+\frac{1}{\left( q+1\right) ^{\frac{1}{q}}}\right)
\left\Vert f^{^{\prime }}\right\Vert _{p}.
\end{eqnarray*}%
The proof is completed.
\end{proof}

\begin{corollary}
Under the assumptions Theorem \ref{3m} with $\alpha =0,$\ we have%
\begin{equation*}
\left\vert f\left( b\right) -\frac{1}{b-a}\int\limits_{a}^{b}f\left(
s\right) ds\right\vert \leq \left( b-a\right) ^{\frac{1}{q}}\left( 1+\frac{1%
}{\left( q+1\right) ^{\frac{1}{q}}}\right) \left\Vert f^{^{\prime
}}\right\Vert _{p}.
\end{equation*}
\end{corollary}

\begin{theorem}
\label{2m} Let $f:\left[ a,b\right] \rightarrow \mathbb{R}$ be a
differentiable function on $\left[ a,b\right] $ and $\left\vert f^{\prime
}\left( x\right) \right\vert \leq M$, for every $x\in $ $\left[ a,b\right] $
and $\alpha \geq 0$. Then the following inequality holds:%
\begin{equation}
\left\vert J_{a}^{\alpha }f\left( b\right) -\frac{\left( b-a\right) ^{\alpha
-1}}{\Gamma \left( \alpha +1\right) }\int\limits_{a}^{b}f\left( s\right)
ds\right\vert \leq \frac{M\left( \alpha +3\right) \left( b-a\right) ^{\alpha
+1}}{2\Gamma \left( \alpha +2\right) }.  \label{7}
\end{equation}
\end{theorem}

\begin{proof}
By use the (\ref{z}), we have%
\begin{eqnarray}
&&\left\vert \Gamma \left( \alpha +1\right) J_{a}^{\alpha }f\left( b\right)
-\left( b-a\right) ^{\alpha -1}\int\limits_{a}^{b}f\left( s\right)
ds\right\vert  \notag \\
&&  \label{8} \\
&\leq &\int\limits_{a}^{b}\left( b-s\right) ^{\alpha }\left\vert f^{^{\prime
}}\left( s\right) \right\vert ds+\left( b-a\right) ^{\alpha
-1}\int\limits_{a}^{b}\left( b-s\right) \left\vert f^{^{\prime }}\left(
s\right) \right\vert ds.  \notag
\end{eqnarray}%
Since $\left\vert f^{^{\prime }}\left( x\right) \right\vert \leq M,$ we get
the required inequality which the proof is completed.
\end{proof}

\begin{corollary}
Under the assumptions Theorem \ref{2m} with $\alpha =0,$\ we have%
\begin{equation*}
\left\vert f\left( b\right) -\frac{1}{b-a}\int\limits_{a}^{b}f\left(
s\right) ds\right\vert \leq \frac{3\left( b-a\right) }{2}M.
\end{equation*}
\end{corollary}

\begin{theorem}
\label{4m} Let $w:\left[ a,b\right] \rightarrow \mathbb{[}0,\infty )$ be a
probability density function, i.e. $\int_{a}^{b}w\left( t\right) dt=1$, and
set $W\left( t\right) =\int_{a}^{t}w\left( x\right) dx$ for $a\leq t\leq b$, 
$W\left( t\right) =0$ for $t<a$ and $W\left( t\right) =1$\ for $t>b$. Let $f:%
\left[ a,b\right] \rightarrow \mathbb{R}$ be a differentiable function on $%
\left[ a,b\right] $ such that $f^{\prime }\in L_{p}\left[ a,b\right] $ with $%
\frac{1}{p}+\frac{1}{q}=1,\ p>1,$ and $\alpha \geq 0$. Then the following
inequality holds:%
\begin{eqnarray}
&&\left\vert \Gamma \left( \alpha +1\right) J_{a}^{\alpha }f\left( b\right)
-\left( b-a\right) ^{\alpha }\int\limits_{a}^{b}w\left( s\right) f\left(
s\right) ds\right\vert  \label{16} \\
&\leq &\left\Vert f^{^{\prime }}\right\Vert _{p}\left( b-a\right) ^{\alpha } 
\left[ \left( \int\limits_{a}^{b}\left\vert W(s)-1\right\vert ^{q}ds\right)
^{\frac{1}{q}}+\left( \frac{b-a}{\alpha q+1}\right) ^{\frac{1}{q}}\right] . 
\notag
\end{eqnarray}
\end{theorem}

\begin{proof}
By using (\ref{hhh}) in (\ref{1}), we have%
\begin{eqnarray}
\Gamma \left( \alpha \right) J_{a}^{\alpha }f\left( b\right)
&=&\int\limits_{a}^{b}\left( b-t\right) ^{\alpha -1}\left[
\int\limits_{a}^{b}w\left( s\right) f\left( s\right)
ds+\int\limits_{a}^{b}P_{w}\left( t,s\right) f^{^{\prime }}\left( s\right) ds%
\right] dt  \notag \\
&&  \notag \\
&=&\int\limits_{a}^{b}\left( b-t\right) ^{\alpha -1}\left(
\int\limits_{a}^{b}w\left( s\right) f\left( s\right) ds\right) dt  \label{y}
\\
&&+\int\limits_{a}^{b}\left( b-t\right) ^{\alpha -1}\left(
\int\limits_{a}^{t}W\left( s\right) f^{^{\prime }}\left( s\right) ds\right)
dt  \notag \\
&&+\int\limits_{a}^{b}\left( b-t\right) ^{\alpha -1}\left(
\int\limits_{t}^{b}\left( W\left( s\right) -1\right) f^{^{\prime }}\left(
s\right) ds\right) dt.  \notag
\end{eqnarray}%
By an interchange of the order of integration, we get%
\begin{equation}
\int\limits_{a}^{b}\left( b-t\right) ^{\alpha -1}\left(
\int\limits_{a}^{b}w\left( s\right) f\left( s\right) ds\right) dt=\frac{%
\left( b-a\right) ^{\alpha }}{\alpha }\int\limits_{a}^{b}w\left( s\right)
f\left( s\right) ds,  \label{11}
\end{equation}%
\begin{equation}
\int\limits_{a}^{b}\left( b-t\right) ^{\alpha -1}\left(
\int\limits_{a}^{t}W\left( s\right) f^{^{\prime }}\left( s\right) ds\right)
dt=\frac{1}{\alpha }\int\limits_{a}^{b}\left( b-s\right) ^{\alpha }W\left(
s\right) f^{^{\prime }}\left( s\right) ds,  \label{120}
\end{equation}%
and%
\begin{eqnarray}
&&\int\limits_{a}^{b}\left( b-t\right) ^{\alpha -1}\left(
\int\limits_{t}^{b}\left( W\left( s\right) -1\right) f^{^{\prime }}\left(
s\right) ds\right) dt  \label{13} \\
&=&\frac{1}{\alpha }\left[ \left( b-a\right) ^{\alpha }\int\limits_{a}^{b}%
\left[ W(s)-1\right] f^{^{\prime }}\left( s\right)
ds+\int\limits_{a}^{b}\left( b-s\right) ^{\alpha }f^{^{\prime }}\left(
s\right) ds\right] .  \notag
\end{eqnarray}%
Thus, using (\ref{11}), (\ref{120}) and (\ref{13}) in (\ref{y}) we get 
\begin{eqnarray}
&&\Gamma \left( \alpha +1\right) J_{a}^{\alpha }f\left( b\right) -\left(
b-a\right) ^{\alpha }\int\limits_{a}^{b}w\left( s\right) f\left( s\right) ds
\label{z1} \\
&=&\left( b-a\right) ^{\alpha }\int\limits_{a}^{b}\left[ W(s)-1\right]
f^{^{\prime }}\left( s\right) ds+\int\limits_{a}^{b}\left( b-s\right)
^{\alpha }f^{^{\prime }}\left( s\right) ds.  \notag
\end{eqnarray}%
By taking the modulus and applying H\"{o}lder inequality, we have%
\begin{eqnarray*}
&&\left\vert \Gamma \left( \alpha +1\right) J_{a}^{\alpha }f\left( b\right)
-\left( b-a\right) ^{\alpha }\int\limits_{a}^{b}w\left( s\right) f\left(
s\right) ds\right\vert \\
&& \\
&\leq &\left( b-a\right) ^{\alpha }\left( \int\limits_{a}^{b}\left\vert
f^{^{\prime }}\left( s\right) \right\vert ^{p}ds\right) ^{\frac{1}{p}}\left(
\int\limits_{a}^{b}\left\vert W(s)-1\right\vert ^{q}ds\right) ^{\frac{1}{q}}
\\
&& \\
&&+\left( \int\limits_{a}^{b}\left\vert f^{^{\prime }}\left( s\right)
\right\vert ^{p}ds\right) ^{\frac{1}{p}}\left( \int\limits_{a}^{b}\left(
b-s\right) ^{\alpha q}ds\right) ^{\frac{1}{q}} \\
&& \\
&=&\left\Vert f^{^{\prime }}\right\Vert _{p}\left( b-a\right) ^{\alpha } 
\left[ \left( \int\limits_{a}^{b}\left\vert W(s)-1\right\vert ^{q}ds\right)
^{\frac{1}{q}}+\left( \frac{b-a}{\alpha q+1}\right) ^{\frac{1}{q}}\right]
\end{eqnarray*}%
which the proof is completed.
\end{proof}

\begin{corollary}
Under the assumptions Theorem \ref{4m} with $\alpha =0,$\ we have%
\begin{equation*}
\left\vert f\left( b\right) -\int\limits_{a}^{b}w\left( s\right) f\left(
s\right) ds\right\vert \leq \left[ \left( \int\limits_{a}^{b}\left\vert
W(s)-1\right\vert ^{q}ds\right) ^{\frac{1}{q}}+\left( b-a\right) ^{\frac{1}{q%
}}\right] \left\Vert f^{^{\prime }}\right\Vert _{p}.
\end{equation*}
\end{corollary}

\begin{theorem}
\label{5m} Let $w:\left[ a,b\right] \rightarrow \mathbb{[}0,\infty )$ be a
probability density function, i.e. $\int_{a}^{b}w\left( t\right) dt=1$, and
set $W\left( t\right) =\int_{a}^{t}w\left( x\right) dx$ for $a\leq t\leq b$, 
$W\left( t\right) =0$ for $t<a$ and $W\left( t\right) =1$\ for $t>b$. Let $f:%
\left[ a,b\right] \rightarrow \mathbb{R}$ be a differentiable function on $%
\left[ a,b\right] $ and $\left\vert f^{\prime }\left( x\right) \right\vert
\leq M$, for every $x\in $ $\left[ a,b\right] $ and $\alpha \geq 0$. Then
the following inequality holds: 
\begin{equation}
\left\vert \Gamma \left( \alpha +1\right) J_{a}^{\alpha }f\left( b\right)
-\left( b-a\right) ^{\alpha }\int\limits_{a}^{b}w\left( s\right) f\left(
s\right) ds\right\vert \leq M\left( b-a\right) ^{\alpha }\left(
\int\limits_{a}^{b}\left\vert W(s)-1\right\vert ds-\frac{b-a}{\alpha +1}%
\right) .  \label{14}
\end{equation}
\end{theorem}

\begin{proof}
From (\ref{z1}), we have%
\begin{eqnarray}
&&\left\vert \Gamma \left( \alpha +1\right) J_{a}^{\alpha }f\left( b\right)
-\left( b-a\right) ^{\alpha }\int\limits_{a}^{b}w\left( s\right) f\left(
s\right) ds\right\vert  \notag \\
&&  \label{15} \\
&\leq &\left( b-a\right) ^{\alpha }\int\limits_{a}^{b}\left\vert
W(s)-1\right\vert \left\vert f^{^{\prime }}\left( s\right) \right\vert
ds+\int\limits_{a}^{b}\left( b-s\right) ^{\alpha }\left\vert f^{^{\prime
}}\left( s\right) \right\vert ds.  \notag
\end{eqnarray}%
By using $\left\vert f^{^{\prime }}\left( x\right) \right\vert \leq M$, the
proof is completed.
\end{proof}

\end{document}